\documentclass[11pt]{article} 
\usepackage{amsmath}
\usepackage{amssymb}
\usepackage{amsthm,mathtools}
\usepackage{lineno}
\usepackage[mathscr]{eucal}
\usepackage{xcolor}
\usepackage{fontenc}
\usepackage{graphicx}
\usepackage{geometry}
\usepackage{lineno}

\theoremstyle{definition}
\newtheorem{definition}{Definition}[section]
\newtheorem{theorem}[definition]{Theorem}
\newtheorem*{theorem*}{Conjecture}

\newtheorem{lemma}[definition]{Lemma}
\newtheorem{corollary}[definition]{Corollary}
\theoremstyle{remark}
\newtheorem{remark}[definition]{Remark}
\newtheorem{example}[definition]{Example}
\newcounter{enumctr}

\newcommand{\R}{\mathbb{R}}
\newcommand{\N}{\mathbb{N}}

\newcommand{\C}{\mathbb{C}}

\begin{document}
	\title {A generalized fractional Halanay inequality and its applications}         
	\author{La Van Thinh\footnote{ lavanthinh@hvtc.edu.vn, Academy of Finance, No. 58, Le Van Hien St., Duc Thang Wrd., Bac Tu Liem Dist., Hanoi, Viet Nam}, 
  Hoang The Tuan\footnote{Institute of Mathematics, Vietnam Academy of Science and Technology, 18 Hoang Quoc Viet, 10307 Ha Noi, Viet Nam}}
	\date{}
	\maketitle
\begin{abstract}
This paper is concerned with a generalized Halanay inequality and its applications to fractional-order delay linear systems. First, based on a sub-semigroup property of Mittag-Leffler functions, a generalized Halanay inequality is established. Then, applying this result to fractional-order delay systems with an order-preserving structure, an optimal estimate for the solutions is given. Next, inspired by the obtained Halanay inequality, a linear matrix inequality is designed to derive the Mittag-Leffler stability of general fractional-order delay linear systems. Finally, numerical examples are provided to illustrate the proposed theoretical results.
\end{abstract}


\section{Introduction}
Fractional delay differential equations is an important class that has many applications in practical problems of fractional differential equations. To our knowledge, the following approaches are commonly used to study the asymptotic behavior of the solutions of these equations: I. Spectrum analysis method; II. Lyapunov--Razumikhin method; III. Comparison method. 

Regarding the spectrum analysis method, interested readers can refer to \cite{Kaslik, Cermak, Phat, TuanIEEE, TuanFCAA}. The drawback of this approach is that it leads to solving complex fractions of fractional orders containing delays and thus requires many tools from complex analysis.

One of the first attempts at formulating a Razumikhin-type theorem for delay fractional differential equations was \cite{Chen}. Recently, this approach has been improved in \cite{Jin, Zhang}. However, the lack of an effective Leibniz rule for fractional derivatives significantly reduces the validity of these results.

Comparison arguments were used very early in fractional calculus, see e.g., \cite{Lak}. They seem to be particularly suitable for positive delay systems \cite{Shen, GCM_20, Jia, BS_22, Tuan-Thinh_23, Thinh-Tuan}.

Halanay inequality is a comparison principle for delay differential equations \cite{Halanay}. In \cite{Wang_15}, the first fractional version of this inequality was established to prove the stability and the dissipativity of fractional-order delay systems. Later, an extended version of \cite[Lemma 2.3]{Wang_15} was proposed in \cite{NNThang} to investigate the finite-time stability of nonlinear fractional order delay systems while other results have been developed in \cite{Tatar, Ke1} (for the case with distributed delays), and in \cite{Ke2} (for the case with unbounded delays).

Motivated by the above discussions, in light of a sub-semigroup property of classical Mittag-Leffler functions, we propose a generalized fractional Halanay inequality which improves and generalizes the existing works \cite[Lemma 2.3]{Wang_15} and \cite[Theorem 1.2]{NNThang}. Then, the obtained inequality is applied to investigate the Mittag-Leffler stability of fractional-order delay systems in both cases: the systems with or without a structure that preserves the order of solutions.

The rest of this paper is organized as follows. In section 2, some preliminaries and a fractional Hanlanay inequality are provided. In section 3, by combining the established fractional Halanay inequality with the property of preserving the order of the solutions, we present a new optimal estimate to characterize the asymptotic stability of fractional-order positive delay linear systems. Next, we consider general fractional-order delay linear systems. With the help of the Halanay-type inequality, a linear matrix inequality is designed to ensure the Mittag-Leffler stability of these systems. In section 4, several numerical examples are presented to illustrate the validity of the theoretical results.

We close this section by introducing some symbols and definitions that will be used throughout the article. Let $\N,\, \R, \, \R_{\ge0},\,\R_+,\;\R_{\le 0}$, $\mathbb C$ be the set of natural numbers, real numbers, nonnegative real numbers, positive real numbers, nonpositive real numbers, and complex numbers, respectively. Let  $d\in\N$ and $\R^d$ stands for the $d$-dimensional real Euclidean space. Denote by $\R^d_{\geq 0}$ the set of all vectors in $\R^d$ with
nonnegative entries, that is, $$\R_{\geq 0}^d=\left\{y=(y_1,...,y_d)^{\rm T}\in \R^d:y_i\ge0,\ 1\le i\le d\right\},$$
$\R^d_{+}$ the set of all vectors in $\R^d$ with
positive entries, that is, $$\R_{+}^d=\left\{y=(y_1,...,y_d)^{\rm T}\in \R^d:y_i>0,\ 1\le i\le d\right\},$$
and $\R^d_{\leq 0}$ the set of all vectors in $\R^d$ with
nonpositive entries, that is, $$\R_{\leq 0}^d=\left\{y=(y_1,...,y_d)^{\rm T}\in \R^d:y_i\le0,\ 1\le i\le d\right\}.$$
For two vectors $u,v\in \R^d$, we write $u\preceq v$ if $u_i\leq v_i$ for all $1\leq i\leq d$. 
Let $A=(a_{ij})_{1\leq i,j\leq d},B=(b_{ij})_{1\leq i,j\leq d}\in \R^{d\times d}$, we write $A\preceq B$ if $a_{ij}\leq b_{ij}$ for all $1\leq i,j\leq d$.
For any $x\in\R^d$, we set $\|x\|:=\displaystyle\sum_{i=1}^d |x_i|$. 
Let $A$ be a matrix in $\R^{d\times d}$.  The transpose of $A$ is denoted by $A^{\rm T}$. The matrix $A$ is Metzler if its off-diagonal entries are nonnegative. It is said to be non-negative if all its entries are non-negative. $A$ is Hurwitz matrix if its spectrum $\sigma{(A)}$ satisfies the stable condition $$\sigma{(A)}\subset\{\lambda\in\C:\Re(\lambda)<0\}.$$
If $x^{\rm T}Ax \le 0,\; \forall x\in\R^d\setminus\{0\}$, the matrix $A$ is negative semi-definite and we write $A\le0$. Given a closed interval $J\subset \R$ and $X$ is a subset of $\R^d$, we define $C(J;\R^d)$ as the set of all continuous functions from $J$ to $X$.

For $\alpha \in (0,1]$ and $T>0$,  the Riemann--Liouville fractional integral of a function $x:[0,T] \rightarrow \mathbb R$ is defined by 
$$ I^\alpha_{0^+}x(t) := \frac{1}{\Gamma(\alpha)}\int_{0}^{t}(t-u)^{\alpha -1}x(u)du,\;t\in (0,T],$$
and its Caputo fractional derivative of the order $\alpha$ as 
$$ ^C D^\alpha_{0^+}x(t) := \frac{d}{dt}I^{1-\alpha}_{0^+}(x(t)-x(0)), \;t\in (0,T],$$
here $\Gamma(\cdot)$ is the Gamma function and $\displaystyle\frac{d}{dt}$ is the usual derivative. For $d\in\N$ and a vector-valued function $x(\cdot)$ in $\R^d,$ we use the notation
$$^{\!C}D^{{\alpha}}_{0^+}x(t):=\left(^{\!C}D^{\alpha}_{0^+}x_1(t),\dots,{^{\!C}D^{\alpha}}_{0^+}x_d(t)\right)^{\rm T}.$$
\section{A generalized fractional Halanay inequality}
In this part, we aim to derive a generalized Halanay-type inequality. To do this, some basic properties of the Mittag-Leffler functions need to be used (especially the sub-semigroup property of the classical Mittag-Leffler functions in Lemma \ref{duoicongtinh} below).

Let $\alpha,\beta\in \R_+$. The Mittag-Leffler function $E_{\alpha,\beta}(\cdot):\R\rightarrow \R$ is defined by 
$$E_{\alpha,\beta} (x):=\sum_{k=0}^\infty \frac{x^k}{\Gamma(\alpha k+\beta)},\;\forall x\in \R.$$ 
When $\beta=1$, for simplicity, we use the convention $E_\alpha(\cdot):=E_{\alpha,1}(\cdot)$ to denote the classical Mittag-Leffler function.

Throughout the rest of the paper, we always assume $\alpha\in (0,1]$.
\begin{lemma} 
	\begin{itemize}
		\item[(i)] $E_{\alpha}(t)>0,\ E_{\alpha,\alpha}(t)>0$ for all $t\in \R$ and $\displaystyle \lim_{t\to \infty}E_{\alpha}(-t)=0$.
		\item[(ii)] $\displaystyle\frac{d}{dt}E_{\alpha}(t)=\frac{1}{\alpha}E_{\alpha,\alpha}(t)$ for all $t\in \R$ and $^{\!C}D^{\alpha}_{0^+}E_{\alpha}(\lambda t^\alpha)=\lambda E_{\alpha}(\lambda t^\alpha)$ for all $\lambda\in \R,\;t\geq 0.$
	\end{itemize}
\end{lemma}
\begin{proof}
\noindent (i) From \cite[Corolary 3.7, p. 29]{Gorenflo}, we have $\displaystyle \lim_{t\to \infty}E_{\alpha}(-t)=0$. The assertions $E_{\alpha}(t)>0,\ E_{\alpha,\alpha}(t)>0$ for all $t\in \R$ are implied from \cite[Proposition 3.23, p. 47]{Gorenflo} and \cite[Lemma 4.25, p. 86]{Gorenflo}.\\
\noindent (ii) By a simple computation, it is easy to check that  $\displaystyle\frac{d}{dt}E_{\alpha}(t)=\frac{1}{\alpha}E_{\alpha,\alpha}(t)$ for all $t\in \R$. The assertion $^{\!C}D^{\alpha}_{0^+}E_{\alpha}(\lambda t^\alpha)=\lambda E_{\alpha}(\lambda t^\alpha)$ for all $\lambda\in \R,\;t\geq 0$ is derived from the fact that the function $E_{\alpha}(\lambda t^\alpha)$ is the unique solution of the initial value problem 
\begin{equation*}
    \begin{cases}
        ^{\!C}D^{\alpha}_{0^+}x(t)&=\lambda x(t),\;t>0,\\
        x(0)&=1.
    \end{cases}
\end{equation*}
\end{proof}
\begin{lemma}(Sub-semigroup property) \cite[Lemma 4]{Ke1}\label{duoicongtinh}
	For $\lambda>0$ and $t,s\ge0$, we have
	\[E_{\alpha}(-\lambda t^{\alpha})E_{\alpha}(-\lambda s^{\alpha})\le E_{\alpha}(-\lambda (t+s)^{\alpha}).\]
\end{lemma}
\begin{lemma} \cite[Lemma 25]{Cong-Tuan-Hieu} \label{compare-FDE}
	Let $x : [0, T] \rightarrow \R$ be continuous and the Caputo fractional derivative
	$^{\!C}D^{\alpha}_{0^+}x(t)$ exists on the interval $(0, T]$. If there exists $t_1\in (0,T]$ such that $x(t_1)=0$ and $x(t)<0,\ \forall t\in [0,t_1)$, then $$^{\!C}D^{\alpha}_{0^+}x(t_1)\geq 0.$$
\end{lemma}
\begin{theorem} \label{Theorem-Hala}
	Let $w:[-\tau,+\infty) \rightarrow \R_{\ge0}$ be continuous functions such that $^{\!C}D^{\alpha}_{0^+}w(\cdot)$ exists on $(0,+\infty)$ and $a(\cdot),\ b(\cdot),\ c(\cdot)$ are nonnegative continuous functions on $[0,+\infty).$ Consider the system
	\begin{align}
		^{\!C}D^{\alpha}_{0^+}w(t)&\le -a(t)w(t)+b(t) \sup_{t-q(t)\le s\le t}w(s)+c(t), \ t>0, \\
		w(s)&=\varphi(s), \ s\in [-\tau,0],
	\end{align}
where $\tau>0$, $\varphi:[-\tau,0] \rightarrow \R_{\geq 0}$ is a given continuous function and the delay function $q:\R_{\geq 0}\rightarrow [0,\tau]$ is continuous. Suppose that  $\displaystyle\sup_{t\ge0}c(t)=c^*$ and one of the following two conditions holds.
\begin{itemize}
	\item [(i)]   $a(\cdot)$ is bounded on the interval $[0,+\infty)$ and $a(t)-b(t)\ge \sigma>0, \ \forall t\ge 0$.
	\item [(ii)]  $a(\cdot)$ is not necessarily bounded on $[0,\infty)$, $a(t)\ge a_0>0, \ \forall t\ge 0$ and $$ \displaystyle\sup_{t\ge0}\frac{b(t)}{a(t)}\le p<1.$$
\end{itemize}
Then, there exists $w_0\ge0,\ \lambda^*>0$ such that
\begin{equation} \label{est_hala}
	w(t)\le w_0+ ME_{\alpha}(-\lambda^*t^\alpha), \ \forall t\ge 0,
\end{equation}
where $M=\displaystyle\sup_{s\in [-\tau,0]}|\varphi(s)|$.
\end{theorem}
\begin{proof}  The proof is divided into three steps.
	
	\textbf{Step 1.} First, we prove that for each fixed $t\geq 0$, there is a unique $\lambda:=\lambda(t)>0$ that satisfies the equation
	\begin{equation} \label{exist_lambda}
		\lambda-a(t)+\frac{b(t)}{E_{\alpha}(-\lambda q^\alpha(t))}=0.	
	\end{equation}
	Indeed, let \[h(\lambda):=\lambda-a(t)+\frac{b(t)}{E_{\alpha}(-\lambda q^\alpha(t))}.\]
	By the fact that $h(\cdot)$ is a continuously differentiable function with respect to the variable $\lambda$ on $[0,+\infty)$, by a simple computation and Lemma 2.1(ii), we obtain
	\[h'(\lambda)=1+\frac{b(t)q^\alpha(t) E_{\alpha,\alpha}(-\lambda q^\alpha(t))}{\alpha\left(E_{\alpha}(-\lambda q^\alpha(t))\right)^2}>0, \ \forall \lambda \in \R_{\geq 0}.\]
	Notice that $h(0)=-a(t)+b(t)<0$ and $\displaystyle\lim_{\lambda\to \infty}h(\lambda)=\infty$. Thus, the equation \eqref{exist_lambda} ($h(\lambda)=0$) has a unique root $\lambda=\lambda(t)\in(0,\infty)$. 
	
	\textbf{Step 2.} Let 
	\[\lambda^*:=\inf_{t\ge0}\left\{\lambda(t):\lambda(t)-a(t)+\frac{b(t)}{E_{\alpha}(-\lambda(t)q^\alpha(t))}=0\right\}.\]
	It is obvious to see $\lambda^*\ge0$. Suppose by contradiction that $\lambda^*=0$. 
	
	Consider the case when the condition (i) is true. There is a $a_1>0$ with $a_1\geq a(t),\ \forall t\ge0.$ From the definition of $\lambda^*$, we can find a $t^1_*\geq 0$ so that $0<\lambda(t^1_*)<\epsilon_1$, where $\epsilon_1$ is small enough satisfying $\epsilon_1<\tilde{p}_1$	and $\tilde{p}_1$ is the unique root of the equation $\tilde{p}_1-\sigma+a_1\displaystyle\left[\frac{1}{E_{\alpha}(-\tilde{p}_1\tau^\alpha)}-1\right]=0.$ Furthermore,
	\begin{align*}
		0&=\lambda(t^1_*)-a(t^1_*)+\frac{b(t^1_*)}{E_{\alpha}(-\lambda(t^1_*)q^\alpha(t^1_*))}\\
		&<\epsilon_1-a(t^1_*)+\frac{a(t^1_*)-\sigma}{E_{\alpha}(-\lambda(t^1_*)q^\alpha(t^1_*))}\\
		&=\epsilon_1-\frac{\sigma}{E_{\alpha}(-\lambda(t^1_*)q^\alpha(t^1_*))}+a(t^1_*)\left[\frac{1}{E_{\alpha}(-\lambda(t^1_*)q^\alpha(t^1_*))}-1\right]\\
		&<\epsilon_1-\sigma+a_1\left[\frac{1}{E_{\alpha}(-\epsilon_1\tau^\alpha)}-1\right]\\
		&<\tilde{p}_1-\sigma+a_1\left[\frac{1}{E_{\alpha}(-\tilde{p}_1\tau^\alpha)}-1\right]
		=0,
	\end{align*}
a contradiction. Here, the final estimate above is derived from strictly increasing to the variable $t$ on $[0,\infty)$ of the function $g_1(\cdot)$ defined by $$g_1(t):=t-\sigma+a_1\displaystyle\left[\frac{1}{E_{\alpha}(-t\tau^\alpha)}-1\right].$$ 

Concerning the assumption (ii), there exists a $t_*^2\geq 0$ such that $0<\lambda(t_*^2)<\epsilon_2$, where $\epsilon_2>0$ is small enough satisfying
	\begin{equation}
		E_{\alpha}(-\epsilon_2 \tau^\alpha)>p \text{ and } \epsilon_2<\tilde{p}_2	
	\end{equation}
	with $\tilde{p}_2$ is the unique root of the equation $\displaystyle\tilde{p}_2-a_0+\frac{pa_0}{E_{\alpha}(-\tilde{p}_2\tau^\alpha)}=0.$
	From the fact that $g_2(t)=t-a_0+\displaystyle\frac{pa_0}{E_{\alpha}(-t\tau^\alpha)}$ is strictly increasing with respect to the variable $t$ on $[0,\infty)$, we conclude
	\begin{align*}
		0&=\lambda(t_*^2)-a(t_*^2)+\frac{b(t_*^2)}{E_{\alpha}(-\lambda(t_*^2)q^\alpha(t_*^2))}\\
		&<\epsilon_2-a(t_*^2)+\frac{pa(t_*^2)}{E_{\alpha}(-\lambda(t_*^2)q^\alpha(t_*^2))}\\
		&=\epsilon_2+a(t_*^2)\left[\frac{p}{E_{\alpha}(-\lambda(t_*^2)q^\alpha(t_*^2))}-1\right]\\
		&<\epsilon_2+a_0\left[\frac{p}{E_{\alpha}(-\epsilon_2\tau^\alpha)}-1\right]\\
		&<\tilde{p}_2+a_0\left[\frac{p}{E_{\alpha}(-\tilde{p}_2\tau^\alpha)}-1\right]
		=0,
	\end{align*}
a contradiction.
	
\textbf{Step 3.} Take
	\[M:=\sup_{s\in[-\tau,0]}|\varphi(s)|.\]
	Assume that (ii) is true. Let $w_0:=\displaystyle\frac{c^*}{(1-p)a_0}\ge0$. To verify the statement \eqref{est_hala}, we first show that
 \begin{equation}\label{addeq1}
     w(t)<w_0+(M+\varepsilon) E_\alpha(-(\lambda^*-\varepsilon)t^\alpha),\;\forall t\geq 0,
 \end{equation}
 where $\varepsilon>0$ is small arbitrarily ($\lambda^*-\varepsilon>0$). Suppose by contradiction that statement \eqref{addeq1} is not true.
Due to $w(0)=\varphi(0)<\displaystyle\frac{c^*}{(1-p)a_0}+M+\varepsilon$, there is a $t_1>0$ such that \begin{align*}
     w(t_1)&=w_0+(M+\varepsilon)E_{\alpha}(-(\lambda^*-\varepsilon)t_1^\alpha),\\
     w(t)&<w_0+(M+\varepsilon)E_{\alpha}(-(\lambda^*-\varepsilon)t^\alpha),\ \forall t\in [0,t_1).
 \end{align*}
	Define \[z(t)=w(t)-w_0-(M+\varepsilon)E_{\alpha}(-(\lambda^*-\varepsilon)t^\alpha),\ t\ge0.\]
	Then,
	\[z(t_1)=0 \text{ and } z(t)<0,\ \forall t\in [0,t_1),\]
	by Lemma \ref{compare-FDE}, it implies that 
 \begin{equation}\label{them1}
 ^{\!C}D^{\alpha}_{0^+}z(t_1)\ge0.
 \end{equation}On the other hand,
\begin{align*}
		{^{\!C}D}_{0^+}^{\alpha}z(t_1)&={^{\!C}D}_{0^+}^{\alpha}w(t_1)+(M+\varepsilon)(\lambda^*-\varepsilon)E_{\alpha}(-(\lambda^*-\varepsilon)t_1^\alpha) \\
		&\le -a(t_1)w(t_1)+b(t_1)\sup_{t_1-q(t_1)\le s\le t_1}w(s)+(M+\varepsilon)(\lambda^*-\varepsilon)E_{\alpha}(-(\lambda^*-\varepsilon)t_1^\alpha)+c^* \\
		&=-w_0a(t_1)-a(t_1)(M+\varepsilon)E_{\alpha}(-(\lambda^*-\varepsilon)t_1^\alpha)+(M+\varepsilon)(\lambda^*-\varepsilon)E_{\alpha}(-(\lambda^*-\varepsilon)t_1^\alpha)\\
  &\hspace{5cm}+b(t_1)\sup_{t_1-q(t_1)\le s\le t_1}w(s)+c^*.
\end{align*}
Noting that $h(\cdot)$ is strictly increasing on $[0,+\infty)$, we have
\[\lambda^*-\varepsilon-a(t_1)+\frac{b(t_1)}{E_{\alpha}(-(\lambda^*-\varepsilon)q^\alpha(t_1))}< \lambda(t_1)-a(t_1)+\frac{b(t_1)}{E_{\alpha}(-\lambda(t_1)q^\alpha(t_1))}.\]
{\bf Case I:} $t_1\le q(t_1)$. It is easy to check that $\displaystyle\sup_{t_1-q(t_1)\le s\le t_1}w(s)<w_0+(M+\varepsilon)$. From this,
\begin{align*}
	^{\!C}D_{0^+}^{\alpha}z(t_1)&<-w_0a(t_1)-a(t_1)(M+\varepsilon)E_{\alpha}(-(\lambda^*-\varepsilon)t_1^\alpha)+(M+\varepsilon)(\lambda^*-\varepsilon)E_{\alpha}(-(\lambda^*-\varepsilon)t_1^\alpha)\\
 &\hspace{3cm}+(M+\varepsilon)b(t_1)+w_0b(t_1)+c^*\\
	&=(M+\varepsilon)E_{\alpha}(-(\lambda^*-\varepsilon)t_1^\alpha)\left[\lambda^*-\varepsilon-a(t_1)+\frac{b(t_1)}{E_{\alpha}(-(\lambda^*-\varepsilon)t_1^\alpha)}\right]\\
 &\hspace{3cm}+a(t_1)\left[w_0\frac{b(t_1)}{a(t_1)}-w_0+\frac{c^*}{a(t_1)}\right]\\
	&\le (M+\varepsilon)E_{\alpha}(-(\lambda^*-\varepsilon)t_1^\alpha)\left[\lambda^*-\varepsilon-a(t_1)+\frac{b(t_1)}{E_{\alpha}(-(\lambda^*-\varepsilon)q^\alpha(t_1))}\right]\\
 &\hspace{3cm}+a(t_1)\left[w_0p-w_0+\frac{c^*}{a_0}\right]\\
	&< (M+\varepsilon)E_{\alpha}(-(\lambda^*-\varepsilon)t_1^\alpha)\left[\lambda(t_1)-a(t_1)+\frac{b(t_1)}{E_{\alpha}(-\lambda(t_1)q^\alpha(t_1))}\right]\\
 &=0, 
\end{align*}
which contracts \eqref{them1}.

{\bf Case 2:} $t_1>q(t_1)$. In this case, we observe that
 \begin{align*}
 \displaystyle\sup_{t_1-q(t_1)\le s\le t_1}w(s)&\le w_0+(M+\varepsilon)\displaystyle\sup_{t_1-q(t_1)\le s\le t_1}E_{\alpha}(-(\lambda^*-\varepsilon)s^\alpha)\\
 &=w_0+(M+\varepsilon)E_{\alpha}(-(\lambda^*-\varepsilon)(t_1-q(t_1))^\alpha).
 \end{align*} 
 This together with Lemma \ref{duoicongtinh} leads to
\begin{align*}
	^{\!C}D^{\alpha}_{0^+}z(t_1)&\le-a(t_1)(M+\varepsilon)E_{\alpha}(-(\lambda^*-\varepsilon)t_1^\alpha)+(M+\varepsilon)(\lambda^*-\varepsilon)E_{\alpha}(-(\lambda^*-\varepsilon)t_1^\alpha)\\
	&\hspace{2.25cm}+b(t_1)E_{\alpha}(-(\lambda^*-\varepsilon)(t_1-q(t_1))^\alpha)+w_0\left[b(t_1)-a(t_1)\right]+c^*\\
	&\le (M+\varepsilon)E_{\alpha}(-(\lambda^*-\varepsilon)t_1^\alpha)\left[\lambda^*-\varepsilon-a(t_1)+\frac{b(t_1)E_{\alpha}(-(\lambda^*-\varepsilon)(t_1-q(t_1))^\alpha)}{E_{\alpha}(-(\lambda^*-\varepsilon)t_1^\alpha)}\right]\\
	&\le (M+\varepsilon)E_{\alpha}(-(\lambda^*-\varepsilon)t_1^\alpha)\left[\lambda^*-\varepsilon-a(t_1)+\frac{b(t_1)}{E_{\alpha}(-(\lambda^*-\varepsilon)q^\alpha(t_1))}\right]\\
	&< (M+\varepsilon)E_{\alpha}(-(\lambda^*-\varepsilon)t_1^\alpha)\left[\lambda(t_1)-a(t_1)+\frac{b(t_1)}{E_{\alpha}(-\lambda(t_1)q^\alpha(t_1))}\right]\\
 &=0,
\end{align*}
a contradiction with \eqref{them1}. In short, we assert that \eqref{addeq1} holds. Let $\varepsilon\to 0$, the estimate \eqref{est_hala} is checked completely.

If the condition (i) is true, choosing $w_0=\displaystyle\frac{c^*}{\sigma}\ge0$ and arguing similarly to the above proof, we also get the desired estimate.
\end{proof}
\begin{remark}
   The theorem \ref{Theorem-Hala} is an extended and improved version of \cite[Lemma 2.3]{Wang_15}, \cite[Lemma 4]{Dongling Wang} and \cite[Theorem 1.2]{NNThang}. 
\end{remark}
\begin{remark}
The key point in the proof of Theorem \ref{Theorem-Hala} is to compare the decay solutions of the original inequality with a given classical Mittag-Leffler function. The difficulty one faces in this situation is that Mittag-Leffler functions in general do not have the semigroup property as exponential functions. Fortunately, the sub-semigroup property (see Lemma \ref{duoicongtinh}) is enough for us to overcome that obstacle.
\end{remark}
Using similar arguments in the proof of Theorem \ref{Theorem-Hala}, we can easily extend this result to the case of various bounded delays as follows.
\begin{corollary} \label{Coro-Hala}
	Let $w:[-\tau,+\infty) \rightarrow \R_+$ be a continuous function such that $^{\!C}D^{\alpha}_{0^+}w(\cdot)$ exists on $(0,+\infty)$ and $a(\cdot),\ b_k(\cdot),\ c(\cdot)$ are nonnegative continuous functions on $[0,+\infty)$, $k=1,\dots,m.$ Consider the system
	\begin{align*}
		^{\!C}D^{\alpha}_{0^+}w(t)&\le -a(t)w(t)+\sum_{k=1}^mb_k(t) \sup_{t-q_k(t)\le s\le t}w(s)+c(t), \ t>0, \\
		w(t)&=\varphi(t), \ t\in [-\tau,0],
	\end{align*}
	where $\varphi:[-\tau,0] \rightarrow \R_+$ is continuous, the delays $q_k(\cdot)$, $k=1,\dots,m$, are continuous and bounded by $\tau$, i.e., $0\le q_k(t)\le \tau,\ \forall t\ge0,\ \forall k=1,\dots,m$. Suppose that  $\displaystyle\sup_{t\ge0}c(t)=c^*$ and one of the following two conditions is true.
	\begin{itemize}
		\item [(C1)]  $a(\cdot)$ is bounded on $[0,+\infty)$, $a(t)-\displaystyle\sum_{k=1}^mb_k(t)\ge \sigma>0, \ \forall t\ge 0$.
		\item [(C2)]  $a(\cdot)$ is not necessarily bounded on $[0,\infty)$, $a(t)\ge a_0>0, \ \forall t\ge 0$ and 
  \[\sup_{t\ge0}\sum_{k=1}^m\frac{b_k(t)}{a(t)}\le p<1.\] 
	\end{itemize}
	Then, there exists $w_0>0,\, \lambda^*>0$ such that
	\begin{equation*}
		w(t)\le w_0+ \sup_{s\in [-\tau,0]}|\varphi(s)|E_{\alpha}(-\lambda^*t^\alpha), \ \forall t\ge 0,
	\end{equation*}
\end{corollary}
where \begin{align*}
    \lambda^*&=\inf_{t\ge0}\left\{\lambda(t):\lambda(t)-a(t)+\displaystyle\sum_{k=1}^m\frac{b_k(t)}{E_{\alpha}(-\lambda(t)q_k^\alpha(t))}=0\right\},\\
    w_0&=\begin{cases}
    \displaystyle\frac{c^*}{\sigma} & \text{in the case when the assumption \textup{(C1)} is satisfied,}\\
    \displaystyle\frac{c^*}{(1-p)a_0} &\text{in the case when the assumption \textup{(C2)} is satisfied.}
\end{cases}
\end{align*}
\section{Mittag-Leffler stability of fractional-order delay linear systems}
\subsection{Fractional-order delay systems with a structure that preserves the order of solutions}
The positive fractional-order system has been studied by many authors before, see e.g., \cite{Shen, GCM_20, Jia, BS_22, Tuan-Thinh_23, Thinh-Tuan}. The method was to use comparison arguments. In the current work, we are concerned with these systems when their initial conditions are arbitrary by exploiting a Halanay-type inequality combined with the property of preserving the order of the solutions. This is a new approach that seems to have never appeared in the literature.

Our research object in this section is the system
\begin{align}
{^{\!C}D}^{{\alpha}}_{0^+}x(t)&=A(t)x(t)+B(t)x(t-q(t)),\, \forall t>0,  \label{Eq-main}\\
		x(t)&=\varphi(t),\; \forall t\in [-\tau,0],  \label{initial-con}
\end{align}
where $A(\cdot)$, $B(\cdot):[0,+\infty) \rightarrow \R^{d\times d}$ are continuous matrix-valued functions, the delay function $q(\cdot):[0,+\infty)  \rightarrow [0,\tau]$ is continuous, and $\varphi(\cdot):[-\tau,0] \rightarrow \R^d$ is a given continuous initial condition. Due to \cite[Theorem 2.2]{Tuan_Hieu}, it can be shown that the initial value problem \eqref{Eq-main}--\eqref{initial-con} has a unique global solution on $[-\tau,+\infty)$ denoted by $\Phi(\cdot,\varphi)$.
\begin{lemma}  \cite[Lemma 2.1]{Thinh-Tuan} \label{Pos-system} Suppose that for each $t\in[0,+\infty),\ A(t)$ is a Metzler matrix and $B(t)$ is a nonnegative matrix. Then, for any initial condition $\varphi(\cdot)\succeq 0$ on $[-\tau,0],$ the solution $\Phi(\cdot,\varphi)$ of the systems \eqref{Eq-main}--\eqref{initial-con} satisfies \[\Phi(\cdot,\varphi)\succeq0 \ \text{on}\ [0,+\infty).\]  
\end{lemma}
\begin{lemma} \label{Pro-order-pre} Consider the system \eqref{Eq-main}. Assume that $A(t)$ is a Metzler Matrix and $B(t)$ is a nonnegative matrix for each $t\ge0$. Let $\varphi,\ \overline{\varphi}\in C([-\tau,0];\R^d)$ with $\varphi(s)\preceq\overline{\varphi}(s),\ \forall s\in[-\tau,0]$. Then, 
\[\Phi(t,\varphi)\preceq\Phi(t,\overline{\varphi}) \ \text{for all}\ t\ge0.\] 
\end{lemma}
\begin{proof}
Define
\[z(t):=\Phi(t,\overline{\varphi})-\Phi(t,\varphi),\;\forall t\geq -\tau.\]
Then,
    \begin{align*}
^{\!C}D^{{\alpha}}_{0^+}z(t)&=^{\!C}D^{{\alpha}}_{0^+}\Phi(t,\overline{\varphi}) -^{\!C}D^{{\alpha}}_{0^+}\Phi(t,\varphi)\\
&=\bigg(A(t)\Phi(t,\overline{\varphi})+B(t)\Phi(t-q(t),\overline{\varphi})\bigg)-\bigg(A(t)\Phi(t,{\varphi})+B(t)\Phi(t-q(t),{\varphi})\bigg)\\
&=A(t)\bigg[\Phi(t,\overline{\varphi})-\Phi(t,{\varphi})\bigg]+B(t)\bigg[\Phi(t-q(t),\overline{\varphi})-\Phi(t-q(t),{\varphi})\bigg]\\
&=A(t)z(t)+B(t)z(t-q(t)),\;\forall t>0,
    \end{align*}
   and
    \[z(s)=\overline{\varphi}(s)-{\varphi}(s)\succeq0\ \text{for all}\ s\in[-\tau,0].\]
From Lemma \ref{Pos-system}, it implies $z(t)\succeq0,\ \forall t\ge0$ or $\Phi(t,\varphi)\preceq\Phi(t,\overline{\varphi}),\ \forall t\ge0.$ The proof is complete.
\end{proof}
\begin{theorem} \label{Theorem-main}
    Consider the system \eqref{Eq-main}--\eqref{initial-con}. Assume that $A(t)$ is Metzler and
$B(t)$ is nonnegative for each $t\ge 0$. In addition, there  exist $a_0>0,\ p\in(0,1) $ satisfying
\begin{equation} \label{con-main} \max_{j\in\{1,\dots,d\}}\sum_{i=1}^da_{ij}(t)\le -a_0\ \ \text{and}\;\;\frac{\displaystyle\max_{j\in\{1,\dots,d\}}\sum_{i=1}^db_{ij}(t)}{\displaystyle\max_{j\in\{1,\dots,d\}}\sum_{i=1}^da_{ij}(t)}\ge -p
\end{equation} 
for all $t\ge0$. Then, for any $\varphi\in C([-\tau,0];\R^d)$, the solution $\Phi(\cdot,\varphi)$ converges to the origin, that is,
\[
\lim_{t\to \infty}\Phi(t,\varphi)=0.
\]
Furthermore, we can find a constant $\lambda>0$ such that
\begin{equation} \label{ets-main}
\|\Phi(t,\varphi)\|\le \bigg(\sup_{s\in [-\tau,0]}\|\varphi(s)\|\bigg)E_{\alpha}(-\lambda t^{\alpha})\ \text{for all}\; t\ge0.
\end{equation}  
\end{theorem}
\begin{proof} 
\textbf{Case 1.} We first take the initial condition $\varphi(\cdot)\in C([-\tau,0];\R^d_{\geq 0})$ on $[-\tau,0].$ To simplify notation, we also denote $x(\cdot)=(x_1(\cdot),\dots,x_d(\cdot))^{\rm T}$ as the solution of system \eqref{Eq-main}--\eqref{initial-con}. By Lemma \ref{Pos-system}, we have $x_i(t)\ge 0$ for all $t\ge0$ and $i=1,\dots,d.$ 


Let
\[X(t):=x_1(t)+x_2(t)+\cdots+x_d(t),\;\forall t\in [-\tau,+\infty).\] 
It is easy to check that
\begin{align*}
    {^{\!C}D}^{{\alpha}}_{0^+}X(t)&={^{\!C}D}^{{\alpha}}_{0^+}x_1(t)+{^{\!C}D}^{{\alpha}}_{0^+}x_2(t)+\cdots+{^{\!C}D}^{{\alpha}}_{0^+}x_d(t)\\
    &\hspace{-1.3cm}=\sum_{j=1}^{d}a_{1j}(t)x_j(t)+\sum_{j=1}^{d}b_{1j}(t)x_j(t-q(t))+\cdots+\sum_{j=1}^{d}a_{dj}(t)x_j(t)+\sum_{j=1}^{d}b_{dj}(t)x_j(t-q(t))\\
    &\hspace{-1.3cm}=\sum_{i=1}^{d}a_{i1}(t)x_1(t)+\sum_{i=1}^{d}a_{i2}(t)x_2(t)+\cdots+\sum_{i=1}^{d}a_{id}(t)x_d(t)\\
    &+\sum_{i=1}^{d}b_{i1}(t)x_1(t-q(t))+\sum_{i=1}^{d}b_{i2}(t)x_2(t-q(t))+\cdots+\sum_{i=1}^{d}b_{id}(t)x_d(t-q(t))\\
    &\hspace{-1.3cm}\le\bigg(\max_{j\in\{1,\dots,d\}}\sum_{i=1}^da_{ij}(t)\bigg)X(t)+\bigg(\max_{j\in\{1,\dots,d\}}\sum_{i=1}^db_{ij}(t)\bigg)X(t-q(t)),\;\forall t>0.
\end{align*}
Let
\[a(t):=-\max_{j\in\{1,\dots,d\}}\sum_{i=1}^da_{ij}(t)\ \text{and}\ b(t):=\max_{j\in\{1,\dots,d\}}\sum_{i=1}^db_{ij}(t)\]
for all $t\geq 0$. It follows from the assumption \eqref{con-main} that $a(t)$ and $b(t)$ satisfy the condition (ii) in Theorem \ref{Theorem-Hala}. This leads to that there exists a $\lambda>0$ such that
\begin{equation}\label{ul1}
0\le X(t)\le \bigg(\sup_{s\in [-\tau,0]}\|\varphi(s)\|\bigg)E_{\alpha}(-\lambda t^{\alpha})\ \text{for all} \ t\ge0.
\end{equation} 

\textbf{Case 2.} Next, let $\varphi(\cdot)\in C([-\tau,0];\R^d_{\leq 0})$. Put $z(t):=-x(t),\ t\ge-\tau$. Then,
\begin{align*}
^{\!C}D^{{\alpha}}_{0^+}z(t)&=-^{\!C}D^{{\alpha}}_{0^+}x(t)=-\bigg(A(t)x(t)+B(t)x(t-q(t))\bigg)\\
&=A(t)z(t)+B(t)z(t-q(t)),\ \forall t>0,\\
z(s)&=-x(s)=-{\varphi}(s)\succeq0,\ \forall s\in[-\tau,0].
    \end{align*}
As shown in {\bf Case 1}, there is a $\lambda>0$ satisfying
\begin{equation*}
0\le z_i(t)\le \bigg(\sup_{s\in [-\tau,0]}\|\varphi(s)\|\bigg)E_{\alpha}(-\lambda t^{\alpha})\ \text{for all} \ t\ge0 \ \text{and}\ i=1,\dots,d,
\end{equation*} 
or
\begin{equation}\label{ul2} 
-\bigg(\sup_{s\in [-\tau,0]}\|\varphi(s)\|\bigg)E_{\alpha}(-\lambda t^{\alpha})\le x_i(t)\le 0\ \text{for all} \ t\ge0 \ \text{and}\ i=1,\dots,d.
\end{equation} 
\textbf{Case 3.} Finally, we consider $\varphi(\cdot)\in C([-\tau,0];\R^d)$. Define $\varphi^+(s):=(\varphi_1^+(s),\dots,\varphi_d^+(s))^{\rm T}$ and $\varphi^-(s):=(\varphi_1^-(s),\dots,\varphi_d^-(s))^{\rm T}$, where, for $i=1,\dots,d$ and $s\in[-\tau,0]$,
\begin{equation*} \varphi^+_i(s)=\begin{cases} \varphi_i(s) &\text{if} \ \varphi_i(s)\geq0,\\ -\varphi_i(s) &\text{if} \ \varphi_i(s)<0,
\end{cases} \ \text{and}\ \varphi_i^-(s)=\begin{cases} \varphi_i(s) &\text{if} \ \varphi(s)\leq 0,\\ -\varphi_i(s) &\text{if} \ \varphi_i(s)>0.
\end{cases}\end{equation*} 
Then, $\varphi^+(\cdot)\in C([-\tau,0];\R^d_{\geq 0}),\ \varphi^-(\cdot)\in C([-\tau,0];\R^d_{\leq 0})$ and 
\[\varphi^-(s)\preceq\varphi(s)\preceq\varphi^+(s)\ \text{for all} \  s\in[-\tau,0].\]
 From Lemma \ref{Pro-order-pre}, we see 
 \begin{equation} \label{est-solu}
 \Phi(t,\varphi^-)\preceq\Phi(t,\varphi)\preceq\Phi(t,\varphi^+) \ \text{for all} \  t\ge0.
 \end{equation}
 Furthermore, from \eqref{ul1} and \eqref{ul2}, we can find $\lambda_1,\lambda_2>0$ satisfying
\begin{align} 
0&\le \Phi_i(t,\varphi^+)\le \bigg(\sup_{s\in [-\tau,0]}\|\varphi^+(s)\|\bigg)E_{\alpha}(-\lambda_1 t^{\alpha})=\bigg(\sup_{s\in [-\tau,0]}\|\varphi(s)\|\bigg)E_{\alpha}(-\lambda_1t^{\alpha}),\label{est_sign1} \\
0&\ge \Phi_i(t,\varphi^-)\ge -\bigg(\sup_{s\in [-\tau,0]}\|\varphi^-(s)\|\bigg)E_{\alpha}(-\lambda_2 t^{\alpha})=-\bigg(\sup_{s\in [-\tau,0]}\|\varphi(s)\|\bigg)E_{\alpha}(-\lambda_2 t^{\alpha}),\label{est_sign2}
\end{align} 
for all $t\ge0$ và $i=1,\dots,d$. By combining \eqref{est-solu}, \eqref{est_sign1} and \eqref{est_sign2}, it leads to  
\begin{equation*} 
-\bigg(\sup_{s\in [-\tau,0]}\|\varphi(s)\|\bigg)E_{\alpha}(-\lambda_2 t^{\alpha})\le \Phi_i(t,\varphi)\le \bigg(\sup_{s\in [-\tau,0]}\|\varphi(s)\|\bigg)E_{\alpha}(-\lambda_1 t^{\alpha})
\end{equation*} 
for all $t\ge0$ và $i=1,\dots,d$, and thus the estimate \eqref{ets-main} is verified with the parameter $\lambda:=\min\{\lambda_1,\lambda_2\}$. In particular, for any $\varphi(\cdot)\in C([-\tau,0];\R^d)$, then 
\[
\lim_{t\to\infty}\Phi(t,\varphi)=0,
\]
which finishes the proof.
\end{proof}
\begin{remark} \label{Remark-bounded}
    Consider the system  \eqref{Eq-main}--\eqref{initial-con}. Suppose that the following assumptions hold.
    \begin{itemize}
        \item[(R1)] $-\displaystyle\max_{j\in\{1,\dots,d\}}\sum_{i=1}^da_{ij}(t)$ is bounded from above on $[0,\infty)$.
        \item [(R2)] $\displaystyle\sup_{t\geq 0}\{\max_{j\in\{1,\dots,d\}}\sum_{i=1}^d a_{ij}(t)+\max_{j\in\{1,\dots,d\}}\sum_{i=1}^db_{ij}(t)\}\le -\sigma$
        with some positive constant $\sigma$.
\end{itemize}
Then, by Theorem \ref{Theorem-Hala}, the conclusions of Theorem \ref{Theorem-main} are still true. 
\end{remark}
\begin{remark}
  Although also established in the class of positive systems like Theorems 4.5, 4.6 in \cite{Thinh-Tuan}, Theorem \ref{Theorem-main} in the current paper provides a new criterion to study the asymptotic behavior of solutions with arbitrary initial conditions.  Indeed, compared to  \cite[Theorem 4.5]{Thinh-Tuan}, Theorem \ref{Theorem-main} does not require the boundedness of the coefficient matrices or the Hurwitz characteristic of the dominant system. Meanwhile, compared to \cite[Theorem 4.5]{Thinh-Tuan}, it is significantly simpler and even allows conclusions about the stability of the systems without having to solve additional supporting inequalities. In section 4, we will show specific numerical examples to clarify these findings.
\end{remark}
\subsection{General fractional-order delay linear systems}
This section deals with general fractional-order delay linear systems. Based on the Halanay inequality established in Theorem \ref{Theorem-Hala}, a linear matrix inequality has been designed to ensure their Mittag-Lefler stability.

Consider the system 
\begin{align}
^{\!C}D^{{\alpha}}_{0^+}x(t)&=A(t)x(t)+B(t)x(t-q(t)),\, \forall t>0,  \label{eq_moi1}\\
		x(t)&=\varphi(t),\; \forall t\in [-\tau,0]. \label{eq_moi2}
\end{align}
Here, $A(\cdot),B(\cdot): [0,\infty)\rightarrow \R^d$ are continuous, $\tau>0$, $q(\cdot):[0,\infty)\rightarrow [0,\tau]$ is a continuous delay function, and $\varphi\in C([-\tau,0];\R^d)$ is an arbitrary initial condition.
\begin{lemma} \cite[Theorem 2]{TuanIET18} \label{Ine-Caputo}
    Let $x:[0,+\infty)\rightarrow \R^d$ is continuous and the Caputo fractional derivative $^{\!C}D^{{\alpha}}_{0^+}x(\cdot)$ exists on $(0,\infty)$. Then, for any $t\ge0$, we have
    \[^{\!C}D^{{\alpha}}_{0^+}\left[x^{\rm T}(t)x(t)\right]\le 2x^{\rm T}(t)^{\!C}D^{{\alpha}}_{0^+}x(t).\]
\end{lemma}
\begin{theorem} \label{Theorem-LMI}
Consider the system \eqref{eq_moi1}--\eqref{eq_moi2}. Suppose that there exist two nonnegative continuous functions
	$\gamma(\cdot),\;\sigma(\cdot):[0,\infty)\rightarrow \R_{\geq 0}$ such that the following linear matrix inequality is satisfied
 \begin{equation} \label{LMI}
	\begin{pmatrix}	[A(t)]^T+A(t)+\gamma(t) I_d & B(t) \\
		[B(t)]^T & -\sigma(t) I_d
	\end{pmatrix}\leq 0,\ \forall t\geq 0,
\end{equation} 
where $I_d$ is the identity matrix in $\R^{d\times d}$. In addition, 
\begin{equation} \label{con-theorem2}
    \gamma(t)\ge a_0>0, \ \forall t\ge 0,\;\text{and} \  \displaystyle\sup_{t\ge0}\frac{\sigma(t)}{\gamma(t)}\le p<1.
\end{equation} Then, there exists a  positive parameter $\lambda>0$ satisfying
	\[\|\Phi(t,\varphi)\|\le \sqrt{\sup_{s\in [-\tau,0]}\|\varphi^{\rm T}(s)\varphi(s)\|}\sqrt{E_{\alpha}(-\lambda t^{\alpha})},\ \forall t\ge 0.\]
\end{theorem}
\begin{proof}
Let $x(\cdot):[-\tau,\infty)\rightarrow \R^d$ be the solution of the system \eqref{eq_moi1}--\eqref{eq_moi2}. Denote $W(t):=x^{\rm T}(t)x(t),\ \forall t\ge -\tau$, then $W(\cdot)$ is a continuous, nonnegative function on $[-\tau,+\infty)$. Using Lemma \ref{Ine-Caputo} and the condition \eqref{LMI}, we have
\begin{align*}
	^{\!C}D^{\alpha}_{0^+}&W(t)+\gamma(t) W(t)-\sigma(t)\sup_{t-q(t)\le s\le t}W(s)\\
	&\le2x^{\rm T}(t)^{\!C}D^{\alpha}_{0^+}x(t)+\gamma(t) x^{\rm T}(t)x(t)-\sigma(t) x^{\rm T}(t-q(t))x(t-q(t))\\
	&=2x^{\rm T}(t)\left[A(t)x(t)+B(t)x(t-q(t))\right]+\gamma(t) x^{\rm T}(t)x(t)-\sigma(t) x^{\rm T}(t-q(t))x(t-q(t))\\
	&=\begin{pmatrix}
		x^{\rm T}(t) & x^{\rm T}(t-q(t))
	\end{pmatrix} \begin{pmatrix}
		[A(t)]^{\rm T}+A(t)+\gamma(t) I_d & B(t) \\
		[B(t)]^{\rm T} & -\sigma(t) I_d
	\end{pmatrix} \begin{pmatrix}
	x(t) \\ x(t-q(t)) 
\end{pmatrix} \\
&\leq 0, \;\;\forall t>0.
\end{align*}
It follows from Theorem \ref{Theorem-Hala} (due to the functions $\gamma(\cdot)$ and $\sigma(\cdot)$ verify the condition \eqref{con-theorem2}) that there is a $\lambda>0$ so that
\[W(t)\le \sup_{s\in [-\tau,0]}\|\varphi^{\rm T}(s)\varphi(s)\|{E_{\alpha}(-\lambda t^{\alpha})},\ \forall t\ge 0.\]
This implies that
\[\|\Phi(t,\varphi)\|\le \sqrt{\sup_{s\in [-\tau,0]}\|\varphi^{\rm T}(s)\varphi(s)\|}\sqrt{E_{\alpha}(-\lambda t^{\alpha})},\ \forall t\ge 0.\]
The proof is complete.
\end{proof}
\begin{remark}
Theorem \ref{Theorem-LMI} is a significant extension of \cite[Proposition 2]{He1}. Furthermore, the convergence rate of the solutions to the origin is also discussed in this result.    
\end{remark}
\begin{remark}
Theorem \ref{Theorem-LMI} is a constructive result. It suggests combining a fractional Halanay inequality with the design of suitable linear matrix inequalities to derive various stability conditions of general delay linear systems.    
\end{remark}
\begin{remark}
    Because the norms on $\R^d$ are equivalent, the correctness of the conclusions in Theorem \ref{Theorem-main} and Theorem \ref{Theorem-LMI} on the asymptotic stability of the systems and the convergence rate of solutions to the origin does not depend on the defined norm. 
\end{remark}
\section{Numerical examples}
This section provides numerical examples to illustrate the validity of the proposed theoretical results.
\begin{example} \label{Example1}
Consider the system
\begin{align}
^{\!C}D^{\alpha}_{0+}x(t)&=-A(t)x(t)+B(t)x(t-q(t)),\; t\in (0,\infty), \label{ex1}\\
y(s)&=\varphi(s),\ s\in[-\tau,0], \label{ex1-cond}
	\end{align}
\end{example}
where $\alpha=0.45$, $\varphi\in C([-\tau,0],\mathbb \R^d)$,
	\begin{align*}
		A(t)&=	\begin{pmatrix}
			-0.7-\displaystyle\frac{1}{\sqrt{1+t}}-0.005t & 1-\displaystyle\frac{1}{\sqrt{1+t}}&0.3+0.2\sin t  \\
			0.1+0.003t & -3-\displaystyle\frac{0.8}{1+t}-0.003t&0.15+0.001t\\
   		0.4+\displaystyle\frac{1}{\sqrt{1+t}} & 1+\displaystyle\frac{0.8}{1+t}+0.001t&-1-0.004t
		\end{pmatrix},\ t\ge0, \\
		B(t)&=	\begin{pmatrix}
			\displaystyle\frac{0.002t^2\sin^2t}{1+t^2} & 0.0015t  & 0\\
			0.0005t & 0.05+\displaystyle\frac{0.1}{2+t}& 0.001t \\0.1 & 0.05-\displaystyle\frac{0.1}{2+t}&\displaystyle\frac{0.12}{3+t}
		\end{pmatrix},\ t\ge 0,
	\end{align*}
and the delay $$q(t)=2-\cos^4t,\ t\geq 0.$$
It is obvious that $\tau=2$. By a simple calculation, we obtain
\begin{align*}
    \max_{j\in\{1,2,3\}}\sum_{i=1}^3a_{ij}(t)&=\max\{-0.2-0.002t,\ -1-0.002t-\displaystyle\frac{1}{\sqrt{1+t}},\ -0.55+0.2\sin t-0.003t\}\\
    &= -0.2-0.002t,\ \forall t\ge 0,\\
    \max_{j\in\{1,2,3\}}\sum_{i=1}^3b_{ij}(t)&=\max\{0.1+0.0005t+\displaystyle\frac{0.002t^2\sin^2t}{1+t^2},\ 0.1+0.0015t,\ 0.001t+\frac{0.12}{3+t}\}\\
    &=0.0015t+0.1,\ \forall t\ge 0.
\end{align*}
This leads to
\begin{align*}
 \max_{j\in\{1,2,3\}}\sum_{i=1}^3a_{ij}(t)&\le-0.2,\ \forall t\ge0,\\
   \frac{\displaystyle\max_{j\in\{1,2,3\}}\sum_{i=1}^3b_{ij}(t)}{\displaystyle\max_{j\in\{1,2,3\}}\sum_{i=1}^3a_{ij}(t)}&=-\frac{0.0015t+0.1}{0.002t+0.2}\ge -0.75,\ \forall t\ge0.
\end{align*}
Thus, the assumptions in Theorem \ref{Theorem-main} are satisfied. From this, for any $\varphi\in C([-2,0];\R^3)$, the solution $\Phi(\cdot,\varphi)$ of the initial value problem \eqref{ex1}--\eqref{ex1-cond} converges to the origin. Choosing 
\[a(t):=0.2+0.002t,\ b(t):=0.1+0.0015t,\ \forall t\ge0.\]
It is easy to check that for $\lambda=0.075$, we have
\begin{align*}
    \lambda-a(t)+\frac{b(t)}{E_{\alpha}(-\lambda q^\alpha(t))}&=-0.125-0.002t+\frac{0.1+0.0015t}{E_{0.45}(-0.075 q^{0.45}(t))}\\
    &\le-0.125-0.002t+\frac{0.1+0.0015t}{E_{0.45}(-0.075\times 2^{0.45})}\\
    &< -0.125-0.002t+\frac{0.1+0.0015t}{0.8}\\
    &=\frac{-0.0001t}{0.8}\\
    &\le0,\ \forall t\ge0.
\end{align*}
Taking 
       $$\varphi(s):=\begin{pmatrix}
        0.2-0.4\cos s \\ 0.1+0.1s \\ \log({s+3})-0.5
\end{pmatrix},\  s\in[-2,0].$$ Because $\displaystyle\sup_{s\in[-2,0]}\|\varphi(s)\|=1.2$, Theorem \ref{Theorem-main} points out that
\[\|\Phi(t,\varphi)\|\le 1.2E_{0.45}(-0.075 t^{0.45}),\ t\ge0.\]
	\begin{figure}
		\begin{center}
			\includegraphics[scale=.7]{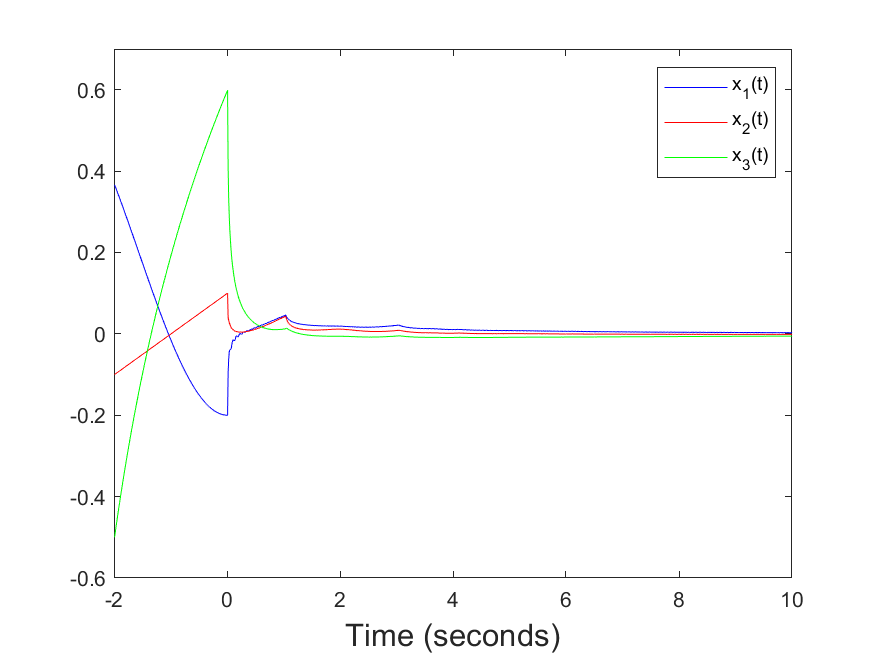}
		\end{center}
		\begin{center}
		\caption{Orbits of the solution of the system \eqref{ex1} with the initial condition $\varphi(s)=(
        0.2-0.4\cos s,0.1+0.1s,\log({s+3})-0.5)^{\rm T}$ on $[-2,0]$.}
		\end{center}
	\end{figure}
 \begin{remark}
     In Example \ref{Example1} above, because the coefficients $a_{ij}(\cdot)$ and $b_{ij}(\cdot)$ are unbounded on $[0,\infty)$, it is outside the scope of \cite[Theorem 4.5]{Thinh-Tuan}. On the other hand, it is extremely complicated to find parameters $\gamma>0$ and $w=(w_1,w_2,w_3)^{\rm T}\in \R^3_+$ that satisfy the following inequalities for all $t\geq 0$:
      \begin{equation*}
		\begin{cases} \label{bpt_1}
			\left(-0.7-\displaystyle\frac{1}{\sqrt{1+t}}-0.005t\right)w_1+\left(1-\displaystyle\frac{1}{\sqrt{1+t}}\right)w_2+\left(0.3+0.2\sin t\right)w_3\\
\hspace{0.5cm}+\displaystyle\frac{0.002t^2\sin^2t}{1+t^2}\frac{w_1}{E_{0.45}(-\gamma2^{0.45})}+\displaystyle \frac{0.0015tw_2}{E_{0.45}(-\gamma2^{0.45})}\hspace{4.2cm}\le-w_1\gamma, \\
\left(0.1+0.003t\right)w_1+\left(-3-\displaystyle\frac{0.8}{1+t}-0.003t\right)w_2+\left(0.15+0.001t\right)w_3\\
   \hspace{0.5cm}+\displaystyle\frac{(0.0005t)w_1}{E_{0.45}(-\gamma2^{0.45})}+\left(0.05+\displaystyle\frac{0.1}{2+t}\right) \frac{w_2}{E_{0.45}(-\gamma2^{0.45})}+\displaystyle\frac{(0.001t)w_3}{E_{0.45}(-\gamma2^{0.45})}\hspace{0.85cm}\le-w_2\gamma, \\
   \left(0.4+\displaystyle\frac{1}{\sqrt{1+t}}\right)w_1+\left(1+\displaystyle\frac{0.8}{1+t}+0.001t\right)w_2+\left(-1-0.004t\right)w_3\\
  \hspace{0.5cm}+\displaystyle\frac{0.1w_1}{E_{0.45}(-\gamma2^{0.45})}+\left(0.05-\displaystyle\frac{0.1}{2+t}\right) \frac{w_2}{E_{0.45}(-\gamma2^{0.45})}+\displaystyle\frac{0.12}{3+t}\frac{w_3}{E_{0.45}(-\gamma2^{0.45})}\le-w_3\gamma.
		\end{cases}
	\end{equation*}
 Therefore, it is not an easy task to test the asymptotic stability and estimate the convergence rate to the origin of the solutions of system \eqref{ex1}--\eqref{ex1-cond} by using \cite[Theorem 4.6]{Thinh-Tuan}.
 \end{remark}
\begin{example} \label{Example2}
    	Consider the system
    \begin{align}
^{\!C}D^{\alpha}_{0+}x(t)&=-A(t)x(t)+B(t)x(t-q(t)),\; t\in (0,\infty), \label{ex2}\\
x(s)&=\varphi(s),\ s\in[-\tau,0], \label{ex2-cond}
	\end{align}
	where $\alpha=0.75$, 
	\begin{align*}
		A(t)=	\begin{pmatrix}
			-3-\displaystyle\frac{1}{\sqrt{1+t}} & 5-\displaystyle\frac{1}{\sqrt{1+t}}  \\
			0.2+\displaystyle\frac{1}{1+t} & -6.6-\displaystyle\frac{0.2}{\sqrt{1+t}} 
		\end{pmatrix},\
		B(t)=	\begin{pmatrix}
			\displaystyle\frac{t\sin^2t}{1+t^2} & 1.15+\displaystyle\frac{0.1}{2+t}  \\
			1.5 & 0.1+\displaystyle\frac{0.2}{2+t}  
		\end{pmatrix},\
	\end{align*}
	and the delay $$q(t)=\frac{1+e^{-t}}{2},\ t\geq 0.$$
 	We see that $\tau=1$ and
 \begin{align*}
     \max_{j\in\{1,2\}}\sum_{i=1}^2a_{ij}(t)&=\max\{-2.8-\displaystyle\frac{1}{\sqrt{1+t}}+\displaystyle\frac{1}{1+t},-1.6-\displaystyle\frac{1.2}{\sqrt{1+t}}\}\\
     &=-1.6-\displaystyle\frac{1.2}{\sqrt{1+t}},\\
     \max_{j\in\{1,2\}}\sum_{i=1}^2b_{ij}(t)&=\max\{1.5+\displaystyle\frac{{t}\sin^2t}{1+t^2};1.25+\displaystyle\frac{0.3}{2+t}\}\\
&=1.5+\displaystyle\frac{t\sin^2t}{1+t^2}.
 \end{align*}
 It easy to check that $\displaystyle\max_{j\in\{1,2\}}\sum_{i=1}^2a_{ij}(t)$ is bounded on $[0,+\infty)$, and
 \begin{align*}\max_{j\in\{1,2\}}\sum_{i=1}^2a_{ij}(t)+\max_{j\in\{1,2\}}\sum_{i=1}^2b_{ij}(t)&=-0.1-\displaystyle\frac{1.2}{\sqrt{1+t}}+\displaystyle\frac{{t}\sin^2t}{1+t^2}\\
 &<-0.1+\displaystyle\frac{{t}}{1+t^2}-\displaystyle\frac{1}{\sqrt{1+t}} \\
 &<-0.1,\ \forall t\ge 0.
 \end{align*}
By Remark \ref{Remark-bounded}, for any $\varphi\in C([-1,0];\R^2)$, the solution $\Phi(\cdot,\varphi)$ of \eqref{ex2} converges to the origin. Taking
\[a(t)=1.6+\displaystyle\frac{1.2}{\sqrt{1+t}},\ b(t)=1.5+\displaystyle\frac{t\sin^2t}{1+t^2},\ t\ge0,\]
and choosing $\lambda=0.02$, we observe
\begin{align*}
    \lambda-a(t)+\frac{b(t)}{E_{\alpha}(-\lambda q^\alpha(t))}&=-1.58-\displaystyle\frac{1.2}{\sqrt{1+t}}+\frac{1.5+\displaystyle\frac{t\sin^2t}{1+t^2}}{E_{0.75}(-0.02q^{0.75}(t))}\\
    &\le-1.58-\displaystyle\frac{1.2}{\sqrt{1+t}}+\frac{1.5+\displaystyle\frac{t}{1+t^2}}{E_{0.75}(-0.02 )}\\
    &<-1.58-\displaystyle\frac{1.2}{\sqrt{1+t}}+\frac{1.5+\displaystyle\frac{t}{1+t^2}}{0.97}\\
    &=-\frac{0.0326}{0.97}+\frac{1}{0.97}\left(\frac{t}{1+t^2}-\frac{1.164}{\sqrt{1+t}}\right)\\
    &<0,\ \forall t\ge0.
\end{align*}
Thus, by Theorem \ref{Theorem-main}, we obtain the estimate $$\|\Phi(t,\varphi)\|\leq \sup_{s\in [-1,0]}\|\varphi(s)\|E_{0.75}(-0.02 t^{0.75}),\;\forall t\geq 0.$$
Figure 2 describes the trajectories of the solution of the initial value problem \eqref{ex2}--\eqref{ex2-cond} with $\varphi(s)=(
        0.3+0.4\sin s,0.1+0.5s)^{\rm T}$ on $[-1,0]$.
	\begin{figure}
		\begin{center}
			\includegraphics[scale=.7]{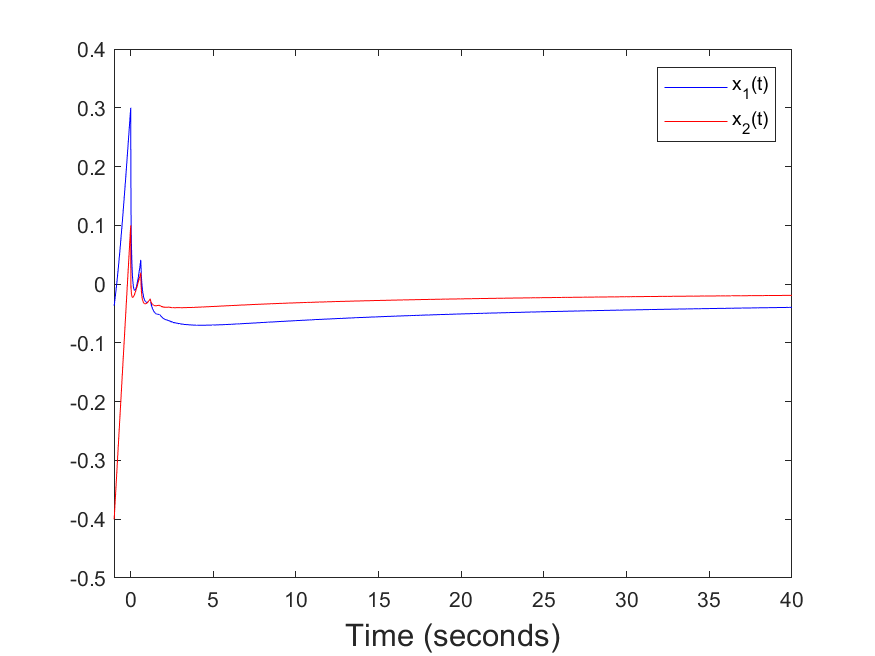}
		\end{center}
		\begin{center}
		\caption{Orbits of the solution of the system \eqref{ex2} with the initial condition $\varphi(s)=(
        0.3+0.4\sin s,0.1+0.5s)^{\rm T}$ on $[-1,0]$.}
		\end{center}
	\end{figure}
\end{example}
\begin{remark}
    	In Example \ref{Example2}, we have
	\begin{align*}
		A(t)\preceq \hat{A}:=\begin{pmatrix}
			-3 & 5  \\
			1.2 & -6.6 
		\end{pmatrix},\ B(t)\preceq \hat{B}:=\begin{pmatrix}
			0.5 & 1.2  \\
			1.5 & 0.2 
		\end{pmatrix},\ \forall t\geq 0.
	\end{align*}
However, $\hat{A}+\hat{B}=\begin{pmatrix}
		-2.5 & 6.2  \\
		2.7 & -6.4 
	\end{pmatrix}$ is not a Hurwitz matrix because $\sigma(\hat{A}+\hat{B})=\{\lambda_1, \lambda_2\}$, here $\lambda_1\approx 0.0824$ and $\lambda_2\approx-8.9824$. Thus, one cannot apply \cite[Theorem 4.5]{Thinh-Tuan} to this case.
\end{remark}
\begin{example}
Consider the system
\begin{align}
^{\!C}D^{\alpha}_{0+}x(t)&=-a(t)x(t)+b(t)x(t-q(t)),\; t\in (0,\infty), \label{ex3}\\
y(s)&=\varphi(s),\ s\in[-\tau,0], \label{ex3-cond}
	\end{align}
\end{example}
where $\alpha=0.65$,
$a(t)=0.2+0.002t,\ b(t)=-0.02\sqrt{t},\;q(t)=1+\displaystyle\frac{1}{2+\sin t}$ for $t\geq 0.$
Taking $\gamma(t)=0.3,\ \sigma(t)=0.2$ for all $t\ge0$, then the condition \eqref{con-theorem2} holds. Moreover,
 \begin{equation*} 
	\begin{pmatrix}
		-2a(t)+\gamma(t)& b(t) \\
		b(t) & -\sigma(t)
	\end{pmatrix}=\begin{pmatrix}
		-0.1-0.004t& -0.02\sqrt{t} \\
		-0.02\sqrt{t} & -0.2
	\end{pmatrix}<0,\ \forall t\geq 0,
\end{equation*}
and thus the condition \eqref{LMI} is also true. Using Theorem \ref{Theorem-LMI}, it shows that the solution $\Phi(\cdot,\varphi)$ converges to the origin for any $\varphi\in C([-2,0];\R)$. Furthermore, by a simple computation, for $\lambda=0.05$, we see
\begin{align*}
    \lambda-\gamma(t)+\frac{\sigma(t)}{E_{\alpha}(-\lambda q^\alpha(t))}&=-0.25+\frac{0.2}{E_{0.65}(-0.05 q^{0.65}(t))}\\
    &\leq -0.25+\frac{0.2}{E_{0.65}(-0.05 \times 2^{0.65})}\\
    &\approx-0.25+\frac{0.2}{0.9179} <0,\ \forall t\ge0.
\end{align*}
Hence, the following estimate is true
$$|\Phi(t,\varphi)|\leq \sqrt{\sup_{s\in [-2,0]}|\varphi(s)|^2}\sqrt{E_{0.65}(-0.05 t^{0.65})},\;\forall t\geq 0.$$
Figure 3 depicts the orbits of the solution of the system \eqref{ex3} with the initial condition $\varphi(s)=0.3-0.5\cos (2s)$ on $[-2,0]$.
	\begin{figure}
		\begin{center}
			\includegraphics[scale=.7]{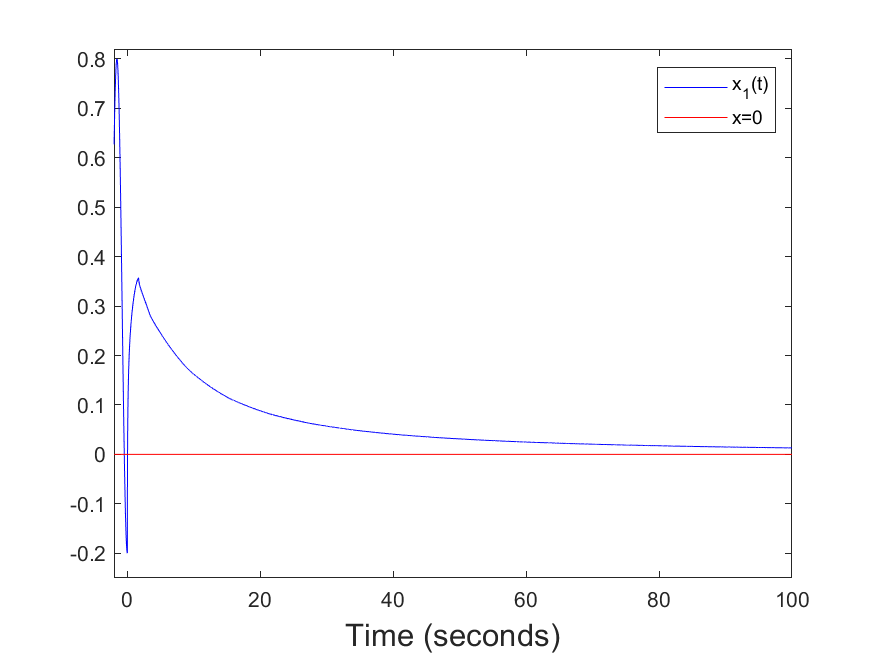}
		\end{center}
		\begin{center}
		\caption{Orbits of the solution of the system \eqref{ex3} with the initial condition $\varphi(s)=0.3-0.5\cos (2s)$ on $[-2,0]$.}
		\end{center}
	\end{figure}

 \end{document}